\documentclass[12pt]{amsart}
\usepackage{amssymb,amsmath,amsthm}
\usepackage[utopia]{mathdesign}
   
   \usepackage{color}

\def\e{{\rm e}}
\def\cic{\mathbf}
\def\eps{\varepsilon}

\def\d{{\rm d}}

\def\R {\mathbb{R}}
\def\N {\mathbb{N}}

\def\D {{\mathcal D}}

\def\I {{\mathcal I}}

\def\F {{\mathcal F}}
\def\Ws {{\mathrm W}}

\def\Tor {{\mathbb T}}

\def\S{{\mathbf S}}
\def\size{{\mathrm{size}}}
\def\dense{{\mathrm{dense}}}
\def\tops{{\mathrm{tops}}}
\def\T{{\mathbf{T}}}
\def\Tor{{\mathbb{T}}}

\def\M {{\mathsf M}}

\def\Wm {{\mathsf W}}

\def \l {\langle}
\def \r {\rangle}

\def \and{\qquad\text{and}\qquad}

\def \no#1#2#3 {{\bf #1} (#3), #2.}
\def \eds#1#2#3 {#1, #2, #3.}
\newtheorem{proposition}{Proposition}
\newtheorem{theorem}{Theorem}

\newtheorem{lemma}[theorem]{Lemma}
\theoremstyle{definition}

\newtheorem*{remark}{Remark}

\numberwithin{theorem}{section}

 \title[Weak-$L^p$ bounds for   Carleson operators]{Weak-$L^p$ bounds for the Carleson and Walsh-Carleson operators}
\author[F.\ Di Plinio]{Francesco Di Plinio}
\address{INdAM - Cofund Marie Curie Fellow at Dipartimento di Matematica, \newline \indent Universit\`a degli Studi di Roma ``Tor Vergata'', \newline  \indent Via della Ricerca Scientifica,   00133 Roma,  Italy   \newline \indent \centerline{and}   \indent
 The Institute for Scientific Computing and Applied Mathematics,
Indiana University
\newline\indent
831 East Third Street, Bloomington, Indiana  47405, U.S.A. }
\email{diplinio@mat.uniroma2.it   }
\thanks{The    author is an INdAM - Cofund Marie Curie Fellow and is  partially
supported by the National Science Foundation under the grant
   NSF-DMS-1206438, and by the Research Fund of Indiana University.}
 
\begin{document}
\begin{abstract}
 
 We prove a weak-$L^p$ bound for  the Walsh-Carleson  operator for $p $ near 1, improving  on a theorem   of Sj\"olin \cite{SJ}. We  relate our result   to the conjectures that the Walsh-Fourier and Fourier series of a function $f\in L\log L(\mathbb T)$ converge for almost every $x \in \Tor$.
 
\end{abstract}
 \maketitle

\section{Motivation and main result}\label{sec1}
The $L^p(\mathbb{T})$, $1<p<\infty$ boundedness of the Carleson maximal operator
$$ \mathrm{C} f(x) =  \sup_{n \in \mathbb N} \bigg|\mathrm{p.v.} \int_{\mathbb T} f(x-t) \e^{2 \pi i n t} \frac{\d t}{t}\bigg|, \qquad x \in \mathbb{T},
$$
first proved in \cite{CARL,H}, entails as a consequence the almost everywhere convergence of the sequence $S_n f$ of partial Fourier sums for each $f \in L^p(\mathbb T)$. A natural question (posed for instance by Konyagin in \cite{K2}) is  whether, given an Orlicz function $\Phi(t)$ such that $L^1(\mathbb T) \subsetneq L^\Phi(\mathbb T) \subsetneq L^p(\mathbb T)$ for all $p>1$, it is true that 
\begin{equation} \label{conj}
 \|\mathrm C f\|_{1,\infty} \leq c \|f\|_{L^\Phi(\mathbb T)},  \end{equation}
 so that (equivalently) $S_n f$ converges almost everywhere to $f$ whenever $f \in L^\Phi(\Tor)$.
It is a result of Antonov \cite{A} that \eqref{conj} holds true for $\Phi(t)=t \log(\e+t)\log\log\log(\e^{\e^\e}+t)$. Antonov's proof makes use  of an approximation technique relying on the smoothness of the Dirichlet kernels to upgrade
 the restricted weak-type estimate of Hunt \cite{H}
 \begin{equation}
\label{hunt}
\|\mathrm{C} \cic{1}_E\|_{ p,\infty } \leq c \frac{p^2}{p-1} |E|^{\frac1p}\qquad \forall\, E\subset \mathbb T, \qquad \forall\, 1<p<\infty,
\end{equation}
to the mixed bound
\begin{equation}
\label{mixed}
\|\mathrm{C} f\|_{ 1,\infty } \leq c\|f\|_{1} \log \Big(\e +\frac{\|f\|_{\infty}}{\|f\|_{1}}\Big),
\end{equation}
which, in turn, yields that $\mathrm{C}: L\log L\log\log\log L(\mathbb T) \to L^{1,\infty}(\mathbb T)$, in view of the log-convexity of the latter space. We remark that a larger quasi-Banach rearrangement invariant space $QA$ such that  $\mathrm{C}: QA \to L^{1,\infty}(\mathbb T)$ holds was found in \cite{ADR}; in \cite{CMR} it is shown that, however, Antonov's space is the largest (in a suitable sense) Orlicz space $L^\Phi(\mathbb T)$ such that the embedding $L^\Phi(\mathbb T)\hookrightarrow QA$ holds. We further note that the results of \cite{A,ADR} have been reproved by Lie \cite{LIE2}, where \eqref{mixed} is obtained directly, without the use of approximation techniques.

The above mentioned results strongly suggest that \eqref{conj} holds for the space $L\log L(\mathbb T)$ as well. If the ``$L\log L$  conjecture'' were true, a consequence would be the  unrestricted version of   Hunt's estimate \eqref{hunt}:\begin{equation}
\label{main}
\|\mathrm{C} f\|_{ p,\infty } \leq  \frac{c}{p-1}\|f\|_{p},\qquad \forall\, 1<p\leq2.
\end{equation}
 On the other hand,   a suitable choice of $p \in (1,2)$ in \eqref{main} yields \eqref{mixed} directly, and in turn, recovers \eqref{conj} for Antonov's $\Phi$; thus, the weak-$L^p$ estimate \eqref{main} arises naturally as an intermediate result between the conjectured $L\log L(\mathbb T)$ bound in \eqref{conj} and the presently known best Orlicz space bound. That the $L\log L$ conjecture implies \eqref{main} is a particular case of the following observation, due to Andrei Lerner (personal communication). Assuming \eqref{conj} holds for a given $\Phi$, one has the pointwise inequality $ M^\#( |\mathrm{C} f|^\frac12) \leq (M_\Phi f)^\frac12$, the latter being the local Orlicz maximal function associated to $\Phi$ \cite[Proposition 5.2]{GMS}. It follows that 
\begin{equation}
\label{andrei}
\| \mathrm{C} f\|_{ p,\infty }\leq c \big\|\big(M^\#( |\mathrm{C} f|^\frac12)\big)^2\big\|_{ p,\infty } \leq c \|M_\Phi f\|_{p,\infty} \leq c \Big(\sup_{t \geq 1} \frac{\Phi(t)}{t^p}\Big)^{\frac1p} \|f\|_p, \qquad \forall\,1<p\leq 2.
\end{equation} 
 Using Antonov's $\Phi(t)=t\log(\e+t) \log\log\log(\e^{\e^\e}+t)$ in \eqref{andrei} leads to 
\begin{equation}
\label{toquote}
\|\mathrm C f\|_{{p,\infty } } \leq  \frac{c}{p-1} \log\log \big(\e^\e+\textstyle \frac{1}{p-1}\big)\|f\|_{p},\qquad \forall 1<p\leq 2;
\end{equation}
to the best of the author's knowledge, there seems to be no better weak-$L^p$ bound than \eqref{toquote} in the current literature, and in particular the validity of \eqref{main}, which can be thought of as a weakening of the $L\log L$ conjecture, is open.

The main new result of this article is that the analogue of \eqref{main} actually holds for the Walsh-Fourier analogue of the Carleson operator, which is often thought of as a discrete model of the Fourier case: see \cite[Chapter 8]{ThWp} for the relevant definitions.
\begin{theorem} \label{walshfull} Denote by $\Ws_n f(x)$ the $n$-th partial Walsh-Fourier sum of $f \in L^1(\mathbb T)$.
 There exists an absolute constant $c>0$ such that
  the Walsh-Carleson maximal operator
  $$
  \Ws  f(x):= \sup_{n \in \N} |\Ws_{n}f(x)|, \qquad x \in \Tor$$
satisfies the operator norm bound
\begin{equation}
\label{wfweak} \|\Ws  \|_{L^{p}(\Tor) \to L^{p,\infty }(\Tor)} \leq \frac{c}{p-1}, \qquad \forall \,1<p\leq 2.
\end{equation}
\end{theorem} 
\begin{remark}[Previous results and sharpness]Theorem \ref{walshfull} is a strengthening of the Walsh analogue of \eqref{hunt}, obtained by Sj\"olin in  \cite{SJ}, and recovers the correspondent version of \eqref{mixed}, first established in \cite{SS}, without the need for approximation techiques developed therein. The bound
$
\Ws:L\log L\log\log\log L(\mathbb T) \to L^{1,\infty}(\mathbb T),
$
which is the Walsh case of Antonov's result, follows as a further consequence.
Furthermore, if we assume that  the Walsh case of the $L\log L$ conjecture is sharp,  in the sense that there exists no Young function  $\Phi$  with $\Ws: L^{\Phi}(\mathbb T) \to L^{1,\infty}(\Tor)$ and such that 
$ 
\limsup_{t\to \infty} (t\log(\e+t))^{-1}{\Phi(t)} =0 
$,
then the bound \eqref{wfweak} is sharp, up to a doubly logarithmic term in $(p-1)^{-1}$;
 see \cite[Section 2]{DP} for details.
\end{remark}
\begin{remark}[The $L\log\log L$ conjecture and weak-$L^p$ bounds for the lacunary Carleson   operator] \label{secrem}
It is conjectured in \cite[Conjecture 3.2]{K2}  that the subsequence $S_{n_j} f$ of the partial Fourier sums of $f \in L\log\log L(\Tor)$ converges almost everywhere whenever $n_{j}$ is a lacunary sequence of integers, in the sense that $n_{j+1}\geq \theta n_j$  for all $j$ and for some $\theta >1$; if true, this result would be sharp. This is equivalent to the conjecture that the \emph{lacunary}  Carleson maximal  operator $$ \mathrm{C}_{\{n_j\}} f(x) =  \sup_{j \in \mathbb N} \bigg|\mathrm{p.v.} \int_{\mathbb T} f(x-t) \e^{2 \pi i n_j t} \frac{\d t}{t}\bigg|, \qquad x \in \mathbb{T},
$$
satisfies
\begin{equation} \label{conjlac}
 \|\mathrm C_{\{n_j\}} f\|_{1,\infty} \leq c \|f\|_{L^\Phi(\mathbb T)},  \end{equation}
for $\Phi(t)=t\log\log(\e^\e+t)$,
with constant $c>0$ depending only on the lacunarity constant $\theta$ of the sequence $\{n_j\}$. 
By  \eqref{andrei}, if  the  above  conjectured bound held true,   the weak-$L^p$ estimate   
\begin{equation}
\label{mainlac}
\|\mathrm{C}_{\{n_j\}} f\|_{ p,\infty } \leq c \log(\e+(p-1)^{-1})\|f\|_{p},\qquad \forall 1<p\leq2
\end{equation}
would follow. The current best result \cite{DP,LIE} is that \eqref{conjlac} holds with   $$\Phi(t)=t\log\log(\e^\e+t)\log\log\log\log(\e^{\cdots^\e}+t).$$
However, we remark that  the argument for the main theorem in \cite{LIE} can be suitably reformulated to  prove  the stronger \eqref{mainlac}  in place of the main result therein (which is an estimate of the same type as \eqref{mixed}, with a $\log\log$ in place of the $\log$). Therefore, the weaker form of Konyagin's $L\log\log L$ conjecture given by \eqref{mainlac} holds true. Finally, we mention that the Walsh analogue of \eqref{mainlac} is explicitly proved in \cite{DP}.
\end{remark}
We give the proof of Theorem \ref{walshfull} in the upcoming Section \ref{secproof}.
For the convenience of the reader, we provide an  appendix, claiming no originality, containing a step-by-step account of the changes needed in  the argument for the main theorem of  \cite{LIE}     to obtain the weak-$L^p$ bound \eqref{mainlac} for the lacunary Carleson operator.
\subsection*{Acknowledgements}
The author is deeply grateful to Andrei Lerner for his contribution and for providing additional motivation for this article. 

\section{Proof of Theorem \ref{walshfull}} \label{secproof}
%
   As usual, we will prove \eqref{wfweak} by relying on the (Walsh) phase plane model sums (see for instance  \cite{Th1,ThWp}). We remark that the main technical tool which is not present in the classical works we   mentioned   is a discrete variant of the  multi-frequency Calder\'on-Zygmund decomposition of \cite{NOT} (Lemma \ref{CZ} below). Similar arguments  have already found ample use in the treatment of discrete modulation-invariant singular integrals \cite{OT,DL,DD2,DP}.   
  
  Let $\D$ be the standard dyadic grid on $\R_+$; below,  we indicate with $\S$ an (arbitrary)  finite collection of \emph{bitiles}, that is rectangles $s=I_s \times \omega_{s}\subset \D \times \D$ with $|\omega_s|=2|I_s|^{-1}$. Denoting by $\omega_{s_1},\omega_{s_2}$, respectively,  the left and right dyadic child  of $\omega_{s}$,  each bitile $s$  is thought of   as the union of the two \emph{tiles} (dyadic rectangles of area 1)   
 $
 s_1 = I_s \times \omega_{s_1},  s_2 = I_s \times \omega_{s_2}$. Writing  $W_{n}$ for the $n$-th Walsh character on $\mathbb T$, the   Walsh wave packet time-frequency adapted to a tile $t=I_t \times \omega_t$ is then defined as 
$$
w_t(x) = \mathrm{Dil}^{2}_{|I_t|} \mathrm{Tr}_{\inf I_t} W_{n_t} (x)=|I_t|^{-1/2}W_{n_t}\Big(\frac{x-\inf I_t}{|I_t|}\Big), \qquad n_t:= |I_t| \inf \omega_t.
$$
The model sums for the Walsh-Carleson maximal operator $\Ws$ are then given by
$$
\Wm_{ \S } f (x) =    \sum_{s \in \S }\eps_s\l f, w_{s_1}\r w_{s_1}(x) \cic{1}_{\omega_{s_2}} (N(x)),
$$
where 
$N:\R^+\to \R^+$ is an (arbitrary) measurable choice function, and $\{\eps_{s}\}\in  \{-1,0,1\}^\S$.   By the reduction  given in e.g.\ \cite{Th1,ThWp}, Theorem \ref{walshfull} is a consequence of    the bound ($p'$ is the H\"older dual of $p$)
\begin{equation}\label{mainprop} 
\|\Wm_\S f\|_{p,\infty} \lesssim  p' \|f\|_p , \qquad \forall\, 1<p\leq 2;
 \end{equation}  
in \eqref{mainprop} and in what follows, the constants implied by the almost inequality signs are meant to be absolute (in particular, independent on $\S$, $N$  and $\{\eps_{s}\}$) and  may vary at each occurrence. 
Observe that   \eqref{mainprop} is recovered by taking $G=\{|\Wm_\S f|>\lambda\}, g(x)= \cic{1}_{G'}(x) \exp(-i \arg(\Wm_\S f(x)))$ in the bound 
\begin{equation}\label{mainprop2} 
|\l \Wm_\S f, g  \r|  
  \lesssim p' \|f\|_p  |G|^{\frac{1}{p'}}, \qquad \forall |g| \leq \cic{1}_{G'},
\end{equation}
where $G' \subset G$ is a suitably chosen (possibly depending on $f$) major subset of $G$: that is, $|G|\leq 4|G'|$.
 By (dyadic) scale-invariance of the family of operators $\{\Wm_\S\}$ over all choices of $\S \subset \D \times \D$ and measurable functions $ N$,  and by linearity in $f$, it suffices to prove \eqref{mainprop2} in the case   $\|f\|_p=1, 1 \leq |G| <4$, to which we  turn in Subsection \ref{ssproof}. In the upcoming Subsection \ref{stuff}, we recall some tools of discrete  time-frequency analysis.
 \subsection{Trees, size and density} \label{stuff}
 We  will use the well-known Fefferman order relation on either tiles or bitiles: $
s \ll s'$ if $I_s \subset I_{s'}$ and $ \omega_{s} \supset \omega_{s'}$
We say that $\S$ is a \emph{convex} collection of bitiles if $s,s'\in \S, s\ll s''\ll s'$ implies $s''\in\S$. It is no restriction to prove \eqref{mainprop} under the further assumption that $\S$ is convex, and we do so. A convex collection of bitiles  $\T\subset \S$  is called \emph{tree} with top bitile $s_\T  $ if $s\ll s_\T$ for all $s \in \T$. 
 We simplify notation and write $I_\T:=I_{s_\T}, \omega_{\T}=\omega_{s_\T}$. We will call \emph{forest}  a collection of (convex) trees  $\T \in \F$, and will make use of the quantity 
$$
\tops(\F) := \sum_{\T \in \F} |I_\T|.
$$

Given a measurable function $N: \R \to \R$  and $G\subset \R$, define
$$
\dense_{G} (\S) = \sup_{s \in  
\S } \sup_{s' \in \S: s \ll s'} \frac{|G\cap I_{s' }\cap N^{-1}(\omega_{s'_2}) |}{|I_{s'}|}.
$$
Furthermore, for  $f \in L^2(\mathbb{T})$, we set
$$
\size_{f} (\S) = \sup_{s \in \mathbf{S}}\max_{j=1,2}   \frac{|\l f,w_{s_j} \r|}{|I_s| ^\frac12}.
 $$
 We observe that $\size,\dense$ are monotone increasing with respect to set inclusion. One has $\dense_G (\S) \leq 1$ for each $G \subset \R$, and it is immediate to see that 
 \begin{equation}
\label{ubsize}
\size_f(\S) \leq \sup_{s\in \S} \inf_{x \in I_s} \M_1 f(x).
\end{equation}
  where $\M_p $, $1\leq p<\infty$,  denotes the (dyadic) $p$-th Hardy-Littlewood maximal function.
Finally, we recall \emph{verbatim} a result from \cite{DL} (Lemma 2.13 therein).
\begin{lemma} \label{tree-est} Let $h \in L^2(\R)$ and $\F$ be a forest with $ \dense_G(\F) \leq \delta,  \,   \tops(\F_\delta) \lesssim \delta^{-1} |G|.$
Then for all $g: \R \to \mathbb C, |g| \leq \cic{1}_G$,
$$ 
|\l \Wm_\F h, g  \r| \lesssim \min\left\{\size_{h}(\F) |G|, \delta^{
\frac12} \sqrt{|G|}{\|h\|_2} \right\}.
$$
\end{lemma}

  \subsection{Proof of \eqref{mainprop2}}\label{ssproof} Recall that we are assuming  $\|f\|_p=1, 1 \leq |G| <4$.
For an appropriate (absolute) choice of   $c>0$,
\begin{equation} \textstyle \label{eset3}\big|E  := \{\M_{p } f \geq c\}  \big|\lesssim c^{-p }\|\M_{p } f \|_{p }^{p }\leq \frac14 .
\end{equation}
Set $G':= G \backslash  E  $; by the above, $|G'| \geq \frac12$, so that $G'$ is a major subset of $G$.
Since $ w_{s_1}(x) \cic{1}_{\omega_{s_2}} (N(x))$ is supported inside $I_s$, we have that $\l w_{s_1},g \r=0$ when  $|g|\leq \mathbf{1}_{G'}$ and $I_s \cap G'=\emptyset$. This means that
\begin{equation} \label{goodtiles}
\l \Wm_\S f, g\r  = \l \Wm_{\S_\mathsf{good}} f, g\r  , \qquad \S_\mathsf{good} := \{s \in \S : I_s \cap E ^c \neq \emptyset\}.
\end{equation}
Therefore, from now on, we will just replace  $\S$ by $\S_\mathsf{good}$ in \eqref{mainprop2}. Note that, as a consequence of \eqref{ubsize} and of the definition of $\S_\mathsf{good},$ we have $\size_f(\S_\mathsf{good})\lesssim 1$.

The next step is to    apply the density decomposition lemma (for instance, \cite[Lemma 2.6]{DL})  to $\S_\mathsf{good}$, writing 
\begin{equation} \label{split}
\S_\mathsf{good} = \bigcup_{\delta \in 2^{-\mathbb N}} \F_{\delta}, \qquad \size_f(\F_{\delta}) \lesssim 1,  \quad \dense_{G} (\F_\delta) \leq \delta, \quad \tops(\F_\delta) \lesssim \delta^{-1} |G|.
\end{equation}
We claim the single forest estimate
\begin{equation}
\label{sfed}
|\l \Wm_{\F_\delta} f, g \r |  \lesssim  \delta^{\frac{1}{p'}}. 
\end{equation}
Assuming \eqref{sfed} holds true,
$$
|\l \Wm_{\S_\mathsf{good}} f, g \r |\leq \sum_{\delta \in 2^{-\N}}|  \l \Wm_{\F_\delta} f, g \r |  \lesssim  \sum_{\delta \in 2^{-\N}} \delta^{\frac{1}{p'}} \lesssim p',
$$
that is, we have proved \eqref{mainprop2}.  The remainder of the section is then devoted to the proof of  the single forest estimate \eqref{sfed}. The key tool is provided by the Lemma below.
\begin{lemma} \label{CZ} For each $\delta \in 2^{-\mathbb N}$, there is a function $h_\delta $ such that $$
\|h_\delta\|_2 \lesssim  \delta^{-\frac12+ \frac{1}{p'} }, \qquad  \l f, w_{s_1} \r = \l  h_\delta, w_{s_1}\r \quad \forall s \in \F_\delta.
$$
\end{lemma}
In particular, we see from  Lemma \ref{CZ} that $\l \Wm_{\F_\delta}  f, g\r= \l \Wm_{\F_\delta}  h_\delta, g\r$ and that $\size_{ h_\delta}(\F_\delta)=\size_{ f}(\F_\delta)\lesssim1$; therefore, we may use   Lemma \ref{tree-est} to bound
$$
|\l \Wm_{\F_\delta}  f, g \r| =| \l \Wm_{\F_\delta} h_{\delta}  , g \r|\lesssim \delta^{\frac12}|G|^{\frac12}\|h_\delta\|_2 \lesssim \delta^{\frac{1}{p'}},
$$
which is \eqref{sfed}. We have thus completed the proof of Theorem \ref{walshfull}, up to showing Lemma \ref{CZ} holds true.
\subsection{Proof of Lemma \ref{CZ}} This argument is analogous to \cite[Lemma 5.1]{DD2}. We argue under the additional assumption that $f$ is supported on $E=\{\M_{p} f \geq c\}$; the general case requires only trivial modifications.
Let $I \in \cic{I}$  be the  maximal dyadic intervals of $E$;  for each $I\in \cic{I}$, let
$ t \in
T_I$ be the collection of all tiles  having $I_t=I$ and which are comparable under $\ll$ to some tile $s_1 \in \F_\delta $. The tiles of $T_I$ are obviously pairwise disjoint. 

The definition of $\S_{\mathsf{good}}$ ensures that, whenever $I_s \cap I\neq \emptyset$ for some $s \in \S_{\mathsf{good}}$ and $I \in \cic{I}$, the inclusion $I \subsetneq I_s$ must hold. It follows that  if $t \in T_I, s_1 \in \{s_1: s \in \T \in \F_\delta\}$ are related under $\ll$, then $t \ll s_1$.
 By standard properties of the Walsh wave packets, 
  $w_{s_1}$   is  a scalar multiple of $w_t$ on $I$;  in particular, $
w_{s_1} \cic{1}_I $ belongs to $H_I$, the subspace of $L^2(I)$ spanned by $\{w_t: t \in T_I\}$. 
A further consequence is that, if $N_I$ is the number of trees $\T \in \F_\delta$ with $I \subset I_\T$, we have $\#T_I \leq N_I$. For $v \in H_I$, we have the inequality
\begin{equation*}
 \|v\|_{L^{p'}(I)}\lesssim N_I^{\frac12-\frac{1}{p'}} \|v\|_{L^2(I)}.
  \end{equation*}  
Since $\|f \|_{L^{p }(I)} \lesssim 1$ by maximality of $I$ in $E$, it then follows that
$$
|(f  ,v)_{L^2( I)}|   \leq   \|f \|_{L^{p }(I)} \|v\|_{L^{p'}(I)} \lesssim N_I^{\frac12-\frac{1}{p'}} \|v\|_{L^2(I)}  \qquad \forall v \in H_I,
$$
and consequently $h_I $, the projection of $f \cic{1}_I$ on $H_I$, satisfies  $\|h_I\|_{L^2(I)}  \lesssim N_I^{\frac12-\frac{1}{p'}}$. Setting
$
h_\delta := \sum_{I \in \cic{I}} h_I,
$
we see that 
$$
\|h_\delta\|_{2}^2 = \sum_{I \in \cic{I}}|I| \|h_I\|_{L^2(I)}^2 \lesssim  \sum_{I \in \cic{I}}|I| N_I^{1-\frac{2}{p'}} \lesssim \Big\|\sum_{\T \in \F_\delta} \cic{1}_{I_\T}\Big\|_{1}^{1-\frac{2}{p'}}\Big( \sum_{I \in \cic{I}} |I|\Big)^{\frac{2}{p'} } \lesssim \delta^{-1+\frac{2}{p'}};$$
in the last step, we made use of the bound on $\tops$ from \eqref{split}, and of \eqref{eset3} to estimate the sum over $I$.
Finally, in view of the above discussion, if $s \in \T\in \F_\delta $   
$$
\l f  , w_{s_1} \r = \sum_{I \in \cic{I}}\l f\cic{1}_I   , w_{s_1}  \r = \sum_{I \in \cic{I}}  \l f  \cic{1}_I  , c w_{t(s_1; I)} \r =   \sum_{I \in \cic{I}}  \l h_I  ,  w_{s_1} \r = \l h_\delta,w_{s_1} \r
$$ 
where $t(s_1; I)$ is the unique (if any) element $t$ of $T_I$ with $t \ll s_1$. This  finishes the proof of the lemma.
\setcounter{section}{0}
\setcounter{proposition}{0}
\renewcommand{\theproposition}{\thesection.\arabic{proposition}}
\setcounter{equation}{0}
\renewcommand{\theequation}{\thesection.\arabic{equation}}
\setcounter{figure}{0}
\setcounter{table}{0}

\appendix

\section{Proof of the weak-$L^p$ bound \eqref{mainlac}}
This appendix is a re-elaboration of the proof of the main theorem of \cite{LIE}, whose content is that the maximal operator $\mathrm C_{\{n_j\}}$ associated to the $\theta$-lacunary sequence $ \{n_j\}$ satisfies 
$$
\|\mathrm C_{\{n_j\}}\|_{1,\infty} \lesssim \|f\|_1 \log\log \Big(\e^\e+ \frac{\|f\|_\infty}{\|f\|_1}\Big).
$$
Our aim is to prove the stronger bound \eqref{mainlac}, that is
$$
\|\mathrm{C}_{\{n_j\}} f\|_{ p,\infty } \leq c \log(\e+(p-1)^{-1})\|f\|_{p},\qquad \forall\,1<p\leq 2;$$
the  $p=2$ case can be obtained by standard Littlewood-Paley theory techniques, so that, by interpolation, it suffices to argue for $1<p\leq4/3$ (say).

 We claim no originality, essentially following step by step the proof in \cite{LIE}, the only difference being that our $(f,\lambda)$ decomposition (in the terminology of \cite[Subsection 3.2]{LIE}) is based on $\M_p$ rather than on $\M_1$, reflecting the assumption $f \in L^p(\mathbb T)$. This allows for the use of (the dual of) Zygmund's inequality in the form\footnote{Here $p'=p/(p-1) \in 2,\infty)$ is the H\"older dual of $p$. }
\begin{equation} \label{ZYG}
\Big(\sum_{k \geq 1} \Big| \int_{\mathbb T}Êf(x)\e^{-i \xi_k x }\, \d x \Big|^2 \Big)^{\frac12} \lesssim p' \|f\|_{L^p(\mathbb T)}
\end{equation}
for each $\theta$-lacunary sequence $\xi_k$ (not necessarily of integers) with $\xi_1\geq 4\theta^{-1}$, and  $1<p\leq 2$, with implicit constant independent on all but $\theta$. The statement given in \eqref{ZYG} above can be found  e.g.\ in \cite{KSZ}: we note that, although the proof in \cite{KSZ} can be modified so that $\sqrt{p'}$ appears in place of $p'$, this is not allowing for any essential improvement in the result we are aiming for.
\subsection{Discretization} Let $\D$ be the standard dyadic grid on $\R$ and $\D_\Tor$ be its restriction to $\Tor$; we indicate by $\S$ the collection of \emph{tiles} $s=I_{s} \times \omega_s \subset \D_\Tor \times \D$ with $|I_s|=|\omega_s|^{-1}$. For  a given   measurable $N$ function on $\Tor$ with range contained in $\{n_j\}$,  and $s \in \S$, set $E(s):=\{x \in I_s: N(x) \in \omega_s\}$. Further,  let  $\psi$ be a smooth function supported on $[2,8]$ and such that 
$$
\sum_{k\geq 0} 2^{k}\psi(2^k x) = \frac{1}{x}, \qquad \forall \,x \in \mathbb{T}\backslash\{0\},
$$
and define
$$
T_s f(x) =\Big( \int_{\Tor} \e^{-i N(x) t} |I_s|^{-1}\psi (|I_s|^{-1}(x-t)) f(t) \,\d t\Big)\cic{1}_{E(s)} (x).
$$
Then, for a suitable choice of $N$ as above, 
$$
|\mathrm C_{\{n_j\}} f(x) | \leq 2 \Big| \sum_{s \in \S} T_s f(x) \Big|, \qquad x \in \mathbb T;
$$
to prove \eqref{mainlac}, it will thus suffice to show that for all measurable $N$ with lacunary range $\{n_j\}$, each $1<p\leq \frac43$, $f \in L^p(\Tor)$ of unit norm and all $G\subset \Tor$ there exists a subset $G'\subset G$ with $|G|\leq 4 |G'|$ such that
\begin{equation}\label{mainlie}
\Big| \Big\l \sum_{s \in \S} T_s f, g \Big\r \Big| \lesssim \log(\e+p')|G|^{\frac{1}{p'}} \qquad \forall |g| \leq \cic{1}_{G'}
\end{equation}
with implicit constant independent of all but (possibly) the lacunarity constant $\theta$. 

Observe that $T_s$ is supported on $E(s) \subset I_s$, and $T^*_s$ is supported on $17I_s\backslash 3I_s$; by further cutting (smoothly) $\psi$ into 32 pieces, we can assume that $T^*_s$ is supported on the interval $I_s^*=I_s+j|I_s|/4$ for some fixed integer $j \in (-40,-8] \cup [8,40)$. Furthermore, there is no loss in generality by working with $j=8$, and we do so, so that $I_s^*=I_s+2|I_s|$ from now on.
\subsection{Proof of \eqref{mainlie}: main reductions}
Let now $f \in L^p(\Tor)$ of unit norm  and $G\subset \Tor$ be given. The first step in the proof of \eqref{mainlie} is the definition of the major set $G'$ as
\begin{equation}
\label{gprime}
G':= G \backslash 1000\{\M_p f(x) \geq c|G|^{-\frac1p}\};
\end{equation}
one obtains that $|G|\leq 4|G'|$ by using the weak-$L^p$ boundedness of $\M_p$ and suitably choosing an absolute constant $c>0$.

The next task is to decompose the collection of tiles $\S$ (roughly) according to  the local $L^p$-norm of $f$ on $I_s$. Define for each $k\geq 0$,
$$
\widetilde{\I_k}:=\{\M_p f(x) \geq c2^{-k}|G|^{-\frac1p}\}=\bigcup_{I \in {\I_k} } I
$$ 
 $I \in \I_k$ being the maximal dyadic intervals of $\widetilde{\I_k}$; we note that, by the maximal theorem and by maximality of $I$
\begin{equation}
\label{equationIk}
|\widetilde{\I_k}|\lesssim 2^{kp} |G|, \qquad \|f\|_{L^p(I)}\lesssim 2^{-k}|G|^{-\frac1p};
\end{equation} 
 the notation ${L^p(I)}$ stands for $L^p(I; \d x/|I|)$. We partition the tiles of $\S$ by making use of the subcollections
$$
\S_{k,o} :=\big\{s_o=I_{s_o} \times \omega_{s_o}: I_{s_o} \in \I_k, 0 \in 2 \omega_{s_o}\big\}, \qquad k  \geq 0
$$
as follows:
\begin{align*}
&\S=  \S_\textrm{clust} \cup \Big( \bigcup_{k \geq 0} \S(k)\Big) \cup \Big( \bigcup_{k \geq 0} \bigcup_{s_o \in \S_{k,o}} \S_{k,1}(s_o) \Big) \cup \Big( \bigcup_{k \geq 0} \bigcup_{s_o \in \S_{k,o}} \S_{k,2}(s_o) \Big),\\
& \S_\textrm{clust} :=\{s: 10\theta \omega_s \ni 0\};\\
					& \S(0):= \{s\not\in \S_\textrm{clust} : I_s^*\subset \widetilde{ \I_{0}}\}, \;\S(k):= \{s\not\in \S_\textrm{clust}: I_s^*\subset  \widetilde{\I_{k}}, I_s^* \cap  \widetilde{\I_{k-1}}=\emptyset \}, \, k=1,2,\ldots,\\
& \S_{k,1}(s_o):=\{s\not\in \S_\textrm{clust}:I_s^*\supset I_{s_o},2\omega_{s}\cap 2\omega_{s_o}\neq \emptyset, \; \textrm{either } I_s^* \cap \widetilde{\I_{k+1}} =\emptyset \textrm{ or } I_s^* \subset \widetilde{\I_{k+1}}\},\\
&\S_{k,2}(s_o):=\{s\not\in \S_\textrm{clust}:I_s^* \supset I_{s_o},2\omega_{s}\cap 2\omega_{s_o}=\emptyset, \; \textrm{either } I_s^* \cap \widetilde{\I_{k+1}} =\emptyset \textrm{ or } I_s^* \subset \widetilde{\I_{k+1}}\}.
\end{align*}
With the above decomposition in hand, \eqref{mainlie} will follow by combining the bounds of the following proposition.  Note that the choice $p\leq \frac43$ guarantees that the summation index exponent $(\frac p2-1)$ appearing in \eqref{badk} below (as well as in the sequel) is uniformly  bounded away from zero.
\begin{proposition}
Let $g$ be a subindicator function supported on $G'$ defined above. We have the estimates
\begin{align}
& \Big| \Big\l \sum_{s \in \S_{\mathrm{clust}}} T_s f, g \Big\r \Big| \lesssim   |G|^{\frac{1}{p'}}, \label{clust}
\\ & \Big\l \sum_{s \in \S(0)} T_s f, g \Big\r =0, \label{bad}
\\ & \Big| \Big\l \sum_{s \in \S(k)} T_s f, g \Big\r \Big| \lesssim  2^{(\frac p2-1)k}|G|^{\frac{1}{p'}}, \qquad k\geq 1,\label{badk}
\\ & \Big| \Big\l\sum_{k\geq 0} \sum_{s_{o}\in \S_{k,o}} \sum_{s \in \S_{k,1}(s_o)} T_s f, g \Big\r \Big| \lesssim  |G|^{\frac{1}{p'}},\label{e1}
\\ & \Big| \Big\l \sum_{k\geq 0} \sum_{s_{o}\in \S_{k,o}} \sum_{s \in \S_{k,2}(s_o)}T_s f, g \Big\r \Big| \lesssim  \log(\e+p')|G|^{\frac{1}{p'}}.\label{e2}
\end{align}
\end{proposition}
\begin{proof}[Proof of \eqref{clust}] This estimate follows from the well-known fact that the operator  $\sum_{s \in \S_{\textrm{clust}}} T_s$, akin to a maximally truncated Hilbert transform, is weak-$L^p$ bounded with operator norm independent of $1<p\leq 2$.
\end{proof}
\begin{proof}[Proof of \eqref{bad}] Note that if $s \in \S(0)$, $I_s$ (the support of $T_s f$) is contained in $30\widetilde{\I_0}$, while $g$ is supported away from $1000 \widetilde{\I_0}$. \end{proof}
\begin{proof}[Proof of \eqref{badk}] Here we use that the support of $\sum_{s \in \S(k)} T_s^* $ is contained in $ \widetilde{\I_k}$, and that    $\|f\|_{L^\infty(I)}\lesssim 2^{-k}|G|^{-\frac1p}$ whenever $I\cap\widetilde{\I_{k-1}}=\emptyset$,  so that
$$
\Big| \Big\l \sum_{s \in \S(k)} T_s f, g \Big\r \Big| \leq 2^{-k}|G|^{-\frac1p}|\widetilde{\I_k}|^{\frac12}\Big\|\sum_{s \in \S(k)} T_s^* g\Big\|_2 \lesssim 2^{(\frac p2-1)k}|G|^{\frac{1}{p'}}, 
$$
using \eqref{equationIk} and the $L^2$-boundedness of $\sum_{s \in \S(k)} T_s^*$, which is essentially an adjoint (discretized) Carleson operator.
\end{proof}
In the next subsection, we give the proof of \eqref{e2}, which  is the only term for which \eqref{ZYG} is needed. The proof of \eqref{e1} requires only minimal modifications from the argument used for \cite[Proposition 3]{LIE}; we omit the details.
\subsection{Proof of \eqref{e2}} We begin with some notation: we write $ A_I(h)$ for the average of $h \in L^1(I)$ on $I$; for a tree $\T$ (see \cite{LIE2} for the definition in this context), we write $(\T)_0$ for the shift of $\T$ to the zero frequency.

We perform a further decomposition of the sum in \eqref{e2}. To begin with, observe that for  fixed $k, s_{o} \in \S_{k,o} $ there exists a $\theta$-lacunary sequence $\{\xi_\ell=\xi_{\ell}(s_0): \ell \geq 1\}$ such that   $\S_{k,2}(s_o)$ can be organized into the union of maximal trees $\T_\ell(s_o)$ each with top frequency $\xi_\ell$, and such that  $\xi_1\geq 4\theta^{-1}|I_{s_o}|^{-1}$ (this point is granted by the requirement $ 2\omega_s \cap 2\omega_{s_o}=\emptyset$ for each $s \in\S_{k,2}(s_o)$). This said, we have that
$$
\sum_{k\geq 0} \sum_{s_{o}\in \S_{k,o}} \sum_{s \in \S_{k,2}(s_o)} T^*_s = \sum_{k\geq 0} \sum_{s_{o}\in \S_{k,o}} \sum_{\ell } T^*_{\T_\ell(s_o)}, \qquad T^*_{\T_\ell(s_o)}:= \sum_{s \in \T_\ell(s_o)} T^*_s,  $$ and   we define
$$
\mathsf{T}_{\T_\ell(s_o)}^* g(x):= \e^{-i \xi_\ell x} A_I \big( T^{*}_{(\T_\ell(s_0))_0} g \big)(x), \qquad \mathsf{R}^*_{\T_\ell(s_o)} g: = T_{\T_\ell(s_o)}^* g - \mathsf{T}_{\T_\ell(s_o)}^* g;
$$
(our $\mathsf{T}$ stands for the notation $T_c$ in \cite{LIE}).
With the above splitting  in hand, the bound \eqref{e2} is obtained by combining the two bounds of the proposition below.
\begin{proposition} \label{bruttoproprio} We have the estimates
\begin{align}
 & \Big| \Big\l f,  \sum_{k\geq 0} \sum_{s_{o}\in \S_{k,o}} \sum_{\ell }  \mathsf{R}^*_{\T_\ell(s_o)} g \Big\r \Big| \lesssim |G|^{\frac{1}{p'}},\label{e21} \\ 
 &  \Big| \Big\l f,  \sum_{k\geq 0} \sum_{s_{o}\in \S_{k,o}} \sum_{\ell}  \mathsf{T}^*_{\T_\ell(s_o)} g \Big\r \Big| \lesssim  \log(\e+p')|G|^{\frac{1}{p'}}.
 \label{e22}
\end{align}
\end{proposition}
\begin{proof}[Proof of \eqref{e21}] Perusal (with minimal changes) of the proofs of \cite[Lemmata 1 and 2]{LIE}.
\end{proof}
\begin{proof}[Proof of \eqref{e22}] We perform here one last decomposition of our set of tiles $\S $. Referring to the mass decomposition recalled in \cite[Subsection 3.1]{LIE} (see also \cite[Section 5]{LIE2}), we write
$
\S= \bigcup_{n \in \mathbb N} \S^{n}, 
$
with (in particular) $|E(s)|\leq 2^{-n}|I_s|$ for all $s 
\in \S^n$. It is important to observe that the mass decomposition above is independent of $f$.

 For each $k,s_o \in \S_{k,o}, $ we set $\S_{k,2}^n(s_o):=\S_{k,2}(s_o)\cap \S^n $; as above, $\S_{k,2}^n(s_o)$ can be partitioned into a union of maximal trees 
${\T_\ell^n(s_o)}$ each with top frequency $\xi_\ell$. 
The sequence $\{\xi_\ell\}$ we obtain is actually a subsequence of the lacunary sequence defined in the previous section: we avoid the subsequence notation for simplicity. 
Estimate \eqref{e22} will be then obtained by summation over $n$ of the inequality
\begin{equation} \label{e222}
\Big| \Big\l f,  \sum_{k\geq 0} \sum_{s_{o}\in \S_{k,o}} \sum_{\ell }  \mathsf{T}^*_{\T_\ell^n(s_o)} g \Big\r \Big| \lesssim \min\{ 1, 2^{-\frac n 2} p'\} |G|^{\frac{1}{p'}}  .
\end{equation}
The left estimate in \eqref{e222} follows by repeating the proof of part (b) of the main theorem in \cite{LIE2} (with minimal modifications). From now on, we devote ourselves to the proof of the right estimate in \eqref{e222}.
Define the square function 
$$
S_ {\S_{k,2}^n(s_o)}( g) (x) := \Big( \sum_{\ell } |T^*_{\T_\ell^n(s_o)} (g) (x) |^2 \Big)^{\frac12}.
$$
It is a consequence of the analysis carried out in \cite{LIE2} that 
\begin{equation}
\label{square}
\Big\|\sum_{s_o \in \S_{k_o}} \cic{1}_{I_{s_o}} S_ {\S_{k,2}^n(s_o)} (g)  \Big\|_2 \lesssim 2^{-\frac n2}|G|^{\frac12}.
\end{equation}
In the same spirit of \cite[Lemma 3]{LIE}, we  have that, for $s_o \in \S_{k,o}$ we have the estimate
\begin{equation}
\label{singleso}
\Big| \Big\l f, \sum_{\ell} \mathsf{T}^*_{\T_\ell^n(s_o)} g \Big\r\Big| \lesssim p'2^{-k}|G|^{-\frac1p}|I_{s_o}|^{\frac12}\| \cic{1}_{I_{s_o}} S_ {\S_{k,2}^n(s_o)}(g) \|_{2}.
\end{equation}
To get \eqref{singleso}, the only modification to the proof of \cite[Lemma 3]{LIE} which we need is bounding  
$$
\Big(\frac{1}{|I_{s_0}|}\sum_{\ell } \Big|\textstyle \int_{I_{s_o}} f(y) \e^{i\xi_\ell y} \d y\Big|^2\Big)^{\frac12} \lesssim p'\|f\|_{L^p(I_{s_o})} \lesssim p' 2^{-k}|G|^{-\frac1p} 
$$
which is the scaled version of inequality \eqref{ZYG}. At this point, taking advantage of  \eqref{singleso},   subsequently making use of Cauchy-Schwarz together with disjointness in $\{s_o \in \S_{k,o}\}$ of the supports of $\cic{1}_{I_{s_o}} S_ {\S_{k,2}^n(s_o)} (g)$, following up  with \eqref{square} and finally relying on    \eqref{equationIk} in the last step,  we obtain
\begin{align*}
  \Big| \Big\l f, \sum_{s_o \in \S_{k_o}} \sum_{\ell} \mathsf{T}^*_{\T_\ell^n(s_o)} g \Big\r\Big| & \lesssim p'2^{-k}|G|^{-\frac1p}  \sum_{s_o \in \S_{k,o}}   |I_{s_o}|^{\frac12} \| \cic{1}_{I_{s_o}} S_ {\S_{k,2}^n(s_o)}(g) \|_{2}
 \\ &
  \leq   p'2^{-k}|G|^{-\frac1p} \Big( \sum_{s_o \in \S_{k,o}} |I_{s_o}| \Big)^{\frac12}  \Big\|\sum_{s_o \in \S_{k_o}} \cic{1}_{I_{s_o}} S_ {\S_{k,2}^n(s_o)} (g)  \Big\|_2 \\ &\lesssim p'2^{-k} 2^{-\frac n2}|G|^{\frac12-\frac1p}|\widetilde{\I_k}|^{\frac12} \lesssim   p' 2^{-\frac n2}|G|^{\frac{1}{p'}} 2^{(\frac p2 -1) k}.
\end{align*}
The  right estimate in \eqref{e222} finally follows by summing up over $k$ the above display.
\end{proof}This concludes the proof of Proposition \ref{bruttoproprio}, and in turn, of \eqref{e2}.
\bibliography{LacunaryCarleson-Fouriercase}{}
\bibliographystyle{amsplain}

\end{document}